\documentclass[12 pt,leqno]{article}

\usepackage{amsmath}
\usepackage{amssymb}
\usepackage[english]{babel}
\usepackage{amsthm}
\usepackage[latin1]{inputenc}
\usepackage{hyperref}
\usepackage{graphicx}
\usepackage{subfigure}
\usepackage{makeidx}

\newtheorem{teo}{Theorem}
\newtheorem{lemma}{Lemma}
\newtheorem{prop}{Proposition}

\newtheorem{remark}{Remark}

\newenvironment{sistema}%
{\left\lbrace\begin{array}{@{}l@{}}}%
{\end{array}\right.}

{\left( \begin{array}{@{}l@{}}}%
{\end{array}\right)}

\begin{document}

\title{\textbf{A mechanical counterexample to KAM theory with low regularity}}
\author{\textbf{Stefano Marò} \\
\textit{\small{Dipartimento di Matematica - Università di Torino}}\\
\textit{\small{Via Carlo Alberto 10, 10123 Torino - Italy}}\\
\textit{\small{e-mail: stefano.maro@unito.it}}
}
\date{}

\maketitle

\begin{abstract}
We give a mechanical example concerning the fact that some regularity is necessary in KAM theory. We consider the model given by the vertical bouncing motion of a ball on a periodically moving plate. Denoting with $f$ the motion of the plate, some variants of Moser invariant curve theorem apply if $\dot{f}$ is small in norm $C^5$ and every motion has bounded velocity. This is not possible if the function $f$ is only $C^1$. Indeed we construct a function $f\in C^1$ with arbitrary small derivative in norm $C^0$ for which a motion with unbounded velocity exists.         
\end{abstract}

%

\section{Introduction}
Moser invariant curve theorem \cite{moserinvcurve} is of fundamental importance to study the stability of the solutions of Hamiltonian systems \cite{moserstable,siegelmoser}. It deals with the existence of invariant curves for some diffeomorphisms of the cylinder that are "close" enough to an integrable twist map. 
More precisely, the map 
\begin{equation}
\begin{sistema}
\theta_1=\theta_{0}+\alpha(r_{0})+R_1(\theta_0,r_0)
\\
r_1=r_{0}+R_2(\theta_0,r_0)
\end{sistema}
\end{equation}
with $\alpha^\prime>0$ has invariant curves if it possesses the intersection property and 
\begin{equation}\label{small}
||R_1||_{C^{333}(\mathcal{C})}+||R_2||_{C^{333}(\mathcal{C})}<\epsilon
\end{equation}
for $\epsilon$ sufficiently small. The assumption on the regularity of the map was very strong. The question of how crucial was this regularity appeared very soon. Takens \cite{takens} gave a first
counterexample in class $C^1$ and successively Herman
\cite{herman1} improved it giving another counterexample in class $C^{3-\varepsilon}$ where $\varepsilon$ is a small positive constant. Recently, Wang \cite{wang} proved a related result for a Hamiltonian system with $d$ degrees of freedom. The examples by Takens and Herman were constructed as perturbations of the integrable map in the class of symplectic diffeomorphisms.   

Our purpose is different, indeed we are going to construct an example in class $C^0$ that comes from a mechanical model. The model describes the vertical motion of a bouncing ball on a moving plate. The plate is moving in the vertical direction as a $1$-periodic function $f$ and the gravity force is acting on the ball. Moreover we suppose that the bounces are elastic and do not affect the movement of the plate. This is a very simple mechanical model with interesting dynamics and has been considered by several authors. See \cite{kunzeortega2,dolgo,qiantorres,holmes} and references therein for more information. The motion of the ball can be described by an exact symplectic twist map that is close to the integrable twist map if the velocity is small \cite{marobounce}. Hence a direct application of Moser's theorem shows that if the velocity of the plate $\dot{f}$ is small in norm $C^{333}$ then invariant curves  exist. It means that the velocity of the ball is always bounded. The smallness of $\dot{f}$ is essential for the boundedness of the velocity. Indeed Pustyl'nikov \cite{pustsoviet} proved that if $\dot{f}$ is sufficiently large then motions of the ball with unbounded velocity exist. We are going to prove that some regularity is needed as well. Precisely given $\delta$ arbitrary small, we are going to construct a concrete function $f\in C^1(\mathbb{R}/\mathbb{Z})$ with $\sup|\dot{f}|\leq\delta$ such that the corresponding model admits a motion of the ball with unbounded velocity. Thus invariant curves cannot exist and Moser's theorem cannot hold in this context. With some care it is possible to adapt our construction so that the function $f$ belongs to a smooth class $C^k (\mathbb{R}/\mathbb{Z})$ with $k\geq 2$ and still it produces unbounded motions. At first sight this could seem contradictory with Moser's theorem but it must be noticed that the $C^k$ norm of this function will be large. An interesting open problem is to determine the optimal $k$ for which all motions have bounded velocity when $||f||_{C^k}$ is sufficiently small. On this line we refer to the work of Zharnitsky \cite{zarpingpong} to see a similar result on the Fermi-Ulam ping-pong model. Our idea of constructing the function is different to Zharnitsky's idea: we start from the result of Pustyl'nikov on the unbounded motion. He constructed an orbit that in the torus $\mathbb{R}/\mathbb{Z}\times\mathbb{R}/\mathbb{Z}$ becomes a fixed point. Our idea will be to look at $N$-cycles in the same torus. This approach will give weaker conditions to generate unbounded orbits. After some technical work, it  will allow us to construct the function $f$ and the corresponding unbounded orbit.

\section{Statement of the problem}
We are concerned with the problem of the motion of a bouncing ball on a vertically moving plate. We assume that the impacts do not affect the motion of the plate that is supposed to move like a function $f\in C^1(\mathbb{R}/\mathbb{Z})$. The linear momentum and the energy are preserved between the bounces, thus the motion is described by the following map    
\begin{equation}\label{unb}
P_f:
\begin{sistema}
t_1=t_{0}+\frac{2}{g}v_{0}-\frac{2}{g}f[t_1,t_{0}]
\\
v_1=v_{0}+2\dot{f}(t_1)-2f[t_1,t_{0}]
\end{sistema}
\end{equation}
where 
$$
f[t_1,t_{0}]=\frac{f(t_1)-f(t_0)}{t_1-t_0}.
$$
Here the coordinate $t$ represents the impact time. The coordinate $v$ represents the velocity of the ball immediately after the impact. This is the formulation considered by Pustil'nikov in \cite{pustsoviet}. Another approach based on differential equations was considered by Kunze and Ortega \cite{kunzeortega2} and leads to a map that is equivalent to (\ref{unb}), see \cite{marobounce}. The map is implicit and is well defined for $v>\bar{v}$ for some $\bar{v}$ sufficiently large. Moreover, by the periodicity of the function $f$, the coordinate $t$ can be seen as an angle. Hence the map $P_f$ is defined on the half cylinder $\mathbb{T}\times (\bar{v},+\infty)$, where $\mathbb{T}=\mathbb{R/\mathbb{Z}}$. \\
If $f\in C^6$, consider the strip $\Sigma_a=\mathbb{T}\times [a,a+k]$ with $a>\bar{v}$ and $k$ sufficiently large. A simple application of Moser invariant curve theorem \cite{moserinvcurve} in the form \cite{ortegaadv} gives the existence of an invariant curve of $P_f$ in $\Sigma_a$ if
\begin{equation}\label{piccolo}
||\dot{f}||_{C^5[0,1]}\leq\delta
\end{equation}
for some $\delta$ sufficiently small. Invariant curves act as barriers so that repeating the argument for $a\to+\infty$ one can prove that if condition (\ref{piccolo}) is satisfied then every orbit $(t_n^*,v_n^*)$ of $P_f$ is such that
$$
\sup_{n\in\mathbb{Z}}v_n^*<\infty.
$$     
This result depends on the regularity of $f$. More precisely we shall prove the following result 
\begin{teo}\label{main}
For every $0<\delta<\frac{g}{4}$ there exists $f\in C^1(\mathbb{R}/\mathbb{Z})$ and an initial condition $(t_0^*,v_0^*)$ such that:
\begin{enumerate}
\item $||\dot{f}||_{C^0[0,1]}\leq\delta$,
\item the orbit of $P_f$ with initial condition $(t_0^*,v_0^*)$ satisfies
$$
t_{n+N}^*=t_n^*+\sigma_n,\quad \sigma_n\in\mathbb{N}
$$
$$
v^*_{n+N}=v_n^*+\frac{g}{2}V \quad\mbox{for some } V\in\mathbb{N}\setminus\{0\}
$$
for every $n\in\mathbb{N}$ and for some $N$ sufficiently large of order $1/\delta$.
\end{enumerate}
\end{teo}  
Throughout the paper $\mathbb{N}$ is the set of non-negative integer numbers (including 0). 

\section{Unbounded orbits}
In this section we are going to construct unbounded orbits for the map $P_f$. We will obtain some intricate conditions that generalize Pustil'nikov's result. The fundamental observation is that the map $P_f$ shares some orbits with a generalized standard map. More precisely, if $(t_n^*,v_n^*)_{n\in\mathbb{Z}}$ is a complete orbit satisfying 
\begin{equation}\label{condug}
f(t_n^*)=f(t_0^*)\quad \mbox{for every $n\in\mathbb{Z}$}
\end{equation}
then $f[t^*_n,t^*_{n-1}]=0$ for every $n\in\mathbb{Z}$ and $(t_n^*,v_n^*)_{n\in\mathbb{Z}}$ becomes a complete orbit for the generalized standard map
\begin{equation}\label{standard}
GS:
\begin{sistema}
t_1=t_{0}+\frac{2}{g}v_{0}
\\
v_1=v_{0}+2\dot{f}(t_1).
\end{sistema}
\end{equation}
The converse is also true: if $(t_n^*,v_n^*)_{n\in\mathbb{Z}}$ is a complete orbit of $GS$ with $v_n>\bar{v}$ for every $n$ and satisfying condition (\ref{condug}) then it is also an orbit for $P_f$. This fact will be crucial in the following.  
We start constructing unbounded orbits for $GS$.
\begin{lemma}\label{stand}
Let $t_0^* < t_1^*$ be real numbers and let $(t_n^*,v_n^*)_{n\in\mathbb{Z}}$ be the orbit of the map $GS$ with initial conditions $t_0=t_0^*, v_0=v^*_0=g(t_1^*-t_0^*)/2$. Suppose that there exist three positive integers $N, W, V$ such that
\begin{enumerate}
\item $N(t_1^*-t_0^*)+\frac{4}{g} \sum_{k=1}^{N-1}(N-k)\dot{f}(t_k^*)=W$,
\item $\frac{4}{g}\sum_{k=0}^{N-1}\dot{f}(t_k^*)=V$.
\end{enumerate}
Then
$$
t_{n+N}^*=t_n^*+\sigma_n,\quad \sigma_n\in\mathbb{N}
$$
$$
v^*_{n+N}=v_n^*+\frac{g}{2}V.
$$
Moreover, there exists $T>0$ such that if $t_1^*-t_0^*>T$ then $v^*_n>\bar{v}$ for every $n\geq 0$. 

\end{lemma}
\begin{proof}
We start noting that from (\ref{standard}) we obtain the following expression for the $n$-th iterate: 
\begin{equation}\label{ast}
v_n=v_0+2\sum_{k=1}^{n}\dot{f}(t_k)
\end{equation}  
\begin{equation}\label{sbiro}
t_n=t_0+\frac{2}{g}nv_0+\frac{4}{g} \sum_{k=1}^{n-1}(n-k)\dot{f}(t_k).
\end{equation}
We claim that for every $j\in\mathbb{N}$, there exists $\sigma_j\in\mathbb{N}$ such that
\begin{equation}\label{funa}
t_{N+j}^*=t_j^*+\sigma_j.
\end{equation}
Let us prove it by induction on $j$. The fact that $v_0^*=g(t_1^*-t_0^*)/2$ and the hypothesis, together with (\ref{sbiro}) give the first step for $j=0$ with $\sigma_0=W$. Note that by periodicity we also have $\dot{f}(t_N^*)=\dot{f}(t_0^*)$.\\
Now suppose that $t_{N+i}^*=t_i^*+\sigma_i$ for every $i<j$. Using (\ref{standard}) we have
\begin{equation}
\begin{split}
&t_{N+j}^*=t^*_{N+j-1}+\frac{2}{g}v^*_{N+j-1}=t^*_{j-1}+\sigma_{j-1}+\frac{2}{g}[v^*_{j-1}+2\sum_{k=0}^{N-1}\dot{f}(t^*_{k+j})]=\\
&(t^*_{j-1}+\frac{2}{g}v^*_{j-1})+\sigma_{j-1}+\frac{4}{g}\sum_{k=0}^{N-1}\dot{f}(t^*_{k+j})=t^*_j+\sigma_{j-1}+\frac{4}{g}\sum_{k=0}^{N-1}\dot{f}(t^*_{k+j}).
\end{split}
\end{equation}
We just need to prove that the last term is an integer. We have that for every $k$, there exist $d\in\mathbb{N}$ and $r\in\{0,\dots,N-1\}$ such that $k+j=Nd+r$. Moreover, the fact that $k\in\{0,\dots,N-1\}$ implies that $N(d-1)+r<j$. This allows to use the inductive hypothesis several times and get
$$
t^*_{k+j}=t^*_{Nd+r}=t^*_{N+N(d-1)+r}=t^*_{N(d-1)+r}+\sigma_{N(d-1)+r}=\dots=t^*_r+\sigma,
$$
where $\sigma\in\mathbb{N}$. Moreover, from the definition, we have that $r$ takes all the values in $\{0,\dots,N-1\}$ as $k$ goes from $0$ to $N-1$. Finally we have
$$
\frac{4}{g}\sum_{k=0}^{N-1}\dot{f}(t^*_{k+j})=\frac{4}{g}\sum_{r=0}^{N-1}\dot{f}(t^*_r)=V
$$
and we conclude using the hypothesis.\\ 
Therefore, from (\ref{ast}), we have
$$
v^*_{N+n}=v^*_n+2\sum_{k=n+1}^{n+N}\dot{f}(t^*_k)=v^*_n+2\sum_{k=0}^{N-1}\dot{f}(t^*_k)=v^*_n+\frac{g}{2}V.
$$
Finally, once more from (\ref{ast}) we have the last assertion remembering that $v_0^*=g(t_1^*-t_0^*)/2$ and $\dot{f}$ is bounded.   
\end{proof}

\begin{remark}
This result has a well-known geometrical interpretation. The map $GS$ satisfies
$$
GS(t_0+1,v_0)=GS(t_0,v_0)+(1,0)
$$
$$
GS(t_0,v_0+\frac{g}{2})=GS(t_0,v_0)+(1,\frac{g}{2}).
$$
It means that $GS$ induces a map on the torus $\mathbb{R}/\mathbb{Z}\times\mathbb{R}/\frac{g}{2}\mathbb{Z}$ and the orbit $(t_n^*,v_n^*)_{n\in\mathbb{Z}}$ becomes an $N$-cycle on this torus.
\end{remark}

We shall use this lemma in the following proposition to find unbounded orbits for the original map $P_f$.

\begin{prop}\label{stand2}
Consider a function $f\in C^1(\mathbb{R}/\mathbb{Z})$ and a sequence $(t_n^*)_{n\in\mathbb{N}}$. Suppose that there exist three positive integers $N, W, V$ such that
\begin{enumerate}
\item $t_N^*-t_0^*=W$,
\item $\frac{4}{g}\dot{f}(t_0^*)+(t_N^*-t_{N-1}^*)-(t_1^*-t_0^*)=V$,
\item $f(t^*_0)=f(t^*_1)=\dots =f(t^*_{N-1})$,
\item $\dot{f}(t_k^*)=\frac{g}{4}(t_{k+1}^*-2t_k^*+t_{k-1}^*)$ for $1\leq k\leq N-1$.
\end{enumerate}
Then if we define $v^*_{n+1}=v^*_{n}+2\dot{f}(t^*_{n+1})$ and $v^*_0=\frac{g(t^*_1-t^*_0)}{2}$ we have that there exists an orbit $(\tau_n^*,\nu_n^*)_{n\in\mathbb{N}}$ of $P_f$ such that $(\tau_n^*,\nu_n^*)=(t_n^*,v_n^*)$ for $0\leq n\leq N$ and  
$$
\tau_{n+N}^*=\tau_n^*+\sigma_n,\quad \sigma_n\in\mathbb{N}
$$
$$
\nu^*_{n+N}=\nu_n^*+\frac{g}{2}V.
$$
Moreover, there exists $T>0$ such that if $t_1^*-t_0^*>T$ then $v^*_n>\bar{v}$ for every $n\geq 0$. 
\end{prop}
\begin{proof}
First of all it is not difficult to prove that conditions 3. and 4. imply that $(t_n^*,v_n^*)$ is a partial orbit of $P_f$ for $0\leq n\leq N$. Note that we get the case $n=N$ using condition 1 and the periodicity of $f$.\\  
Hence, to prove our result, it is sufficient to prove that hypothesis 1,2 and 4 allow to apply Lemma \ref{stand}. Indeed the sequence $(t_n)$ coming from Lemma 1 satisfies condition (\ref{funa}). Using hypothesis 3 we have that condition (\ref{condug}) holds and we can repeat the discussion of the beginning of this section.

Let us prove that from hypothesis 2 and 4 we can recover condition 2 in Lemma \ref{stand}. We just have to verify that
$$
(t_N^*-t_{N-1}^*)-(t_1^*-t_0^*)=\frac{4}{g}\sum_{k=1}^{N-1}\dot{f}(t_k^*)
$$ 
and, remembering hypothesis 4, it is sufficient to prove that 
\begin{equation}\label{fds}
(t_N^*-t_{N-1}^*)-(t_1^*-t_0^*)=\sum_{k=1}^{N-1}T_k.
\end{equation}
Here, we denote 
\begin{equation}\label{tk}
T_k=t_{k+1}^*-2t_k^*+t_{k-1}^*.
\end{equation}
Let us prove (\ref{fds}) by induction on $N$. The base case $N=1$ can be easily verified. Now suppose as inductive hypothesis (\ref{fds}) to be true. Using it we have
\begin{equation}
\sum_{k=1}^{N}T_k = \sum_{k=1}^{N-1}T_k+t_{N+1}^*-2t_N^*+t_{N-1}^* = (t_{N+1}^*-t_{N}^*)-(t_1^*-t_0^*)
\end{equation} 
that proves our claim. To get condition 1 of Lemma \ref{stand} we note that, from hypothesis 1 we have
\begin{equation}
\begin{split}
W &= t_N^*-t_0^*=t_N^*-t_0^*+N(t_1^*-t_0^*)-N(t_1^*-t_0^*) \\
  &= N(t_1^*-t_0^*)+(N-1)t_0^*-Nt_1^*+t_N^*.
\end{split}
\end{equation}
Once again using hypothesis 4 we just have to prove that
\begin{equation}\label{tfe}
(N-1)t_0^*-Nt_1^*+t_N^*=\sum_{k=1}^{N-1}[T_k(N-k)]
\end{equation}
where $T_k$ is defined by (\ref{tk}).
Let us prove it by induction on $N$. The base case $N=1$ can be easily verified. Now suppose as inductive hypothesis (\ref{tfe}) to be true. Simple computations give
\begin{equation}
\begin{split}
 \sum_{k=1}^{N}[T_k(N+1-k)]&=\sum_{k=1}^{N-1}[T_k(N+1-k)]+T_N\\
 & =\sum_{k=1}^{N-1}[T_k(N-k)]+\sum_{k=1}^{N-1}T_k+T_N.
\end{split}
\end{equation}
Using the inductive hypothesis and the definition of $T_N$ we get
$$
\sum_{k=1}^{N}[T_k(N+1-k)]=(N-1)t_0^*-Nt_1^*-t_N^*+t_{N+1}^*+t_{N-1}^*+\sum_{k=1}^{N-1}T_k.
$$
Now we can use (\ref{fds}) and get
$$
\sum_{k=1}^{N}[T_k(N+1-k)]=Nt_0^*-(N+1)t_1^*+t_{N+1}^*.
$$
So we can recover also condition 1 in Lemma \ref{stand} and conclude the proof.
\end{proof}

\section{Proof of Theorem \ref{main}}

Proposition \ref{stand2} gives conditions to decide whether a finite sequence $(t_n)_{0\leq n<N}$ "generates" an unbounded orbit of $P_f$. In this section we are going to construct a sequence $(t_n)$ and a function $f$ in such a way that Proposition \ref{stand2} is applicable. The next lemma deals with the construction of the sequence.    

\begin{lemma}\label{stand3}
For every $\delta\in (0,\frac{g}{4})$ there exist three positive integers $N,W,V$ and an increasing sequence $(t_n)_{0\leq n\leq N}$ satisfying the following conditions.
\begin{enumerate}
\item $t_N-t_0=W$,
\item $\frac{4}{g}\eta+(t_N-t_{N-1})-(t_1-t_0)=V$ for some $0<\eta\leq\delta$,
\item $\frac{g}{4}(t_{n+1}-2t_n+t_{n-1})=\delta$ for $1\leq n\leq N-1$.
\end{enumerate}
\end{lemma}
\begin{proof}
We construct the sequence $(t_n)$ for $0\leq n \leq N$ for some $N$ to be fixed later. Fix $t_0=0$ and consider $t_1$ positive to be fixed later. Define, for every $0\leq n \leq N-1$ the increasing sequence
\begin{equation}\label{deftn}
t_{n+1}=\frac{4}{g}\delta+2t_n-t_{n-1}
\end{equation}
so that condition 3. is satisfied. Now let us adjust the constants $t_1,N,W,V$ and $\eta$ in order to satisfy conditions 1. and 2. Let us start by noticing that letting $t_0=0$, the formula
\begin{equation}\label{deftnn}
t_n=\frac{n(n-1)}{2}\frac{4}{g}\delta+nt_1
\end{equation}
holds for every $n\geq 0$ and $t_1>0$.
We use it to rewrite condition 1. as
\begin{equation}\label{un}
Nt_1+N(N-1)\frac{2\delta}{g}=W
\end{equation}
and condition 2. as
\begin{equation}\label{do}
\frac{4}{g}\eta+(N-1)\frac{4}{g}\delta=V
\end{equation}   
Now we just have to find $N,V,W\in\mathbb{N}\setminus \{0\}$, $t_1>0$ and $0<\eta\leq \delta$ such that (\ref{un}) and (\ref{do}) are satisfied. First consider (\ref{do}). Fix $V=1$ so that (\ref{do}) is equivalent to
\begin{equation}\label{defeta}
\eta=\frac{g}{4}-(N-1)\delta.
\end{equation}
Imposing $0<\frac{g}{4}-(N-1)\delta\leq \delta$ we get the condition
$$
\frac{g}{4\delta}\leq N<\frac{g}{4\delta}+1
$$ 
that is satisfied by some $N>1$. Using such $N$ we can define $\eta$ through (\ref{defeta}). Now we can consider (\ref{un}). We have 
$$
t_1=\frac{W}{N}-N(N-1)\frac{2\delta}{g}.
$$ 
If we chose $W=N^2(N-1)$ we can conclude noting that
$$
t_1=N(N-1)(1-\frac{2\delta}{g})>0.
$$
\end{proof}

The following proposition is concerned with the construction of the function $f$.

\begin{prop}\label{prop2}
Consider a pair of sequences $(t_k)_{0\leq k \leq N}$ and $(D_k)_{0\leq k \leq N}$ such that $t_k\leq t_{k+1}$ and $0\leq D_k\leq \delta$ for some $\delta>0$. 
Suppose that $t_N-t_0=W$ for some $W\in\mathbb{N}$ and $D_0=D_N$. 
Then there exists $f\in C^1(\mathbb{R}/\mathbb{Z})$ such that
\begin{enumerate}
\item $f(t_0)=f(t_1)=\dots =f(t_{N-1})$
\item $\dot{f}(t_k)=D_k$ for $1\leq k\leq N$
\item $||\dot{f}||_{C^0[0,1]}\leq\delta$
\end{enumerate} 
\end{prop}
\begin{proof}
To fix the ideas, suppose that $t_0=0$. Consider the sequence $(t_k)_{0\leq k\leq N}$ modulo $1$ given by
\begin{equation*}
\begin{sistema}
t_k\mapsto t_k-[t_k]\quad \mbox{for } 0\leq k\leq N-1 \\
t_N\mapsto 1
\end{sistema}
\end{equation*}
where $[x]$ represents the integer part of $x$. We can rearrange the sequence supposing it to be monotone non-decreasing. To be consistent we will rearrange also the sequence $(D_k)$ following the permutation made on the sequence $(t_k)$. 
Now for $t\in[0,1]$ consider the function $\zeta(t)$ being piecewise linear defined for $t_k\leq t< t_{k+1}$, $0\leq k<N$ as in Figure \ref{etaa}.
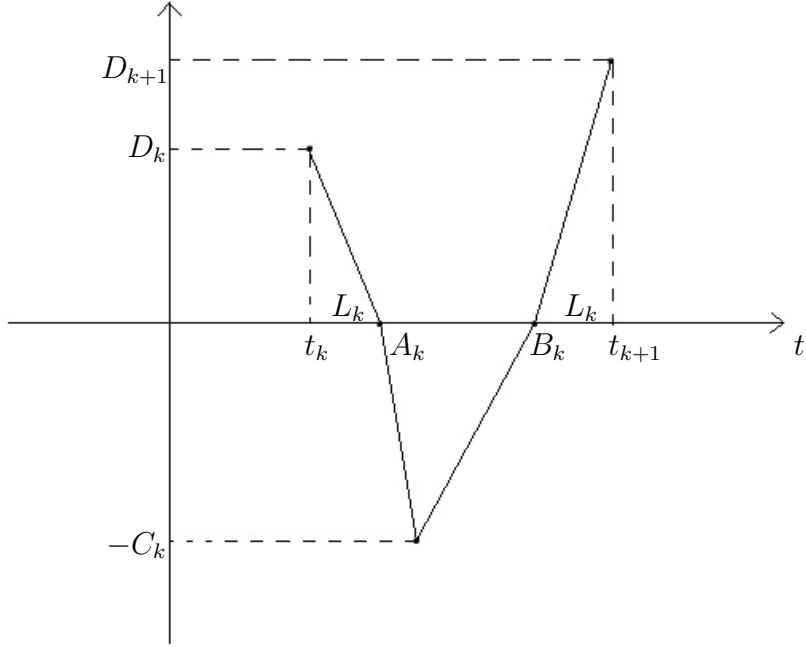
\begin{figure}[h]\label{etaa}
\begin{center}
\includegraphics[scale=0.8]{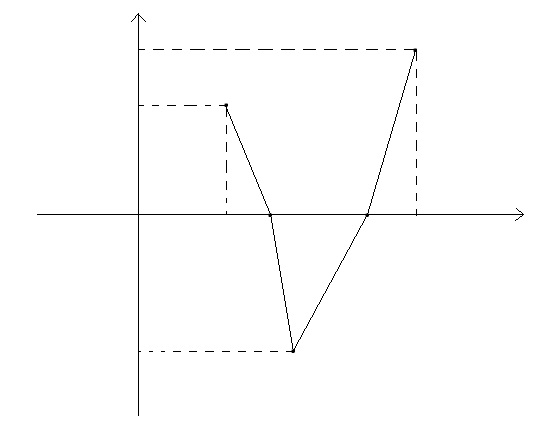}
\put(-10,115){$t$}
\put(-110,115){$B_k$}
\put(-80,115){$t_{k+1}$}
\put(-163,115){$A_k$}
\put(-270,40){$-C_k$}
\put(-195,115){$t_k$}
\put(-185,130){$L_k$}
\put(-272,220){$D_{k+1}$}
\put(-262,190){$D_{k}$}
\put(-97,130){$L_k$}
\end{center}
 \caption{The function $\zeta(t)$ for $t_k\leq t< t_{k+1}$}
 \end{figure}
With reference to the figure, the points $A_k$ and $B_k$ are determined by the positive quantity $L_k< \frac{t_{k+1}-t_k}{2}$ and the constant $C_k$ is such that $0<C_k<\delta$. If we were able to get the signed area between $t_k$ and $t_{k+1}$ to be zero, we would get the thesis extending $\zeta(t)$ to the whole $\mathbb{R}$ by periodicity and letting
$$
f(t)=\int_0^t\zeta(s)ds.
$$
We will prove that it is possible to construct such a function $\zeta$ finding suitable $C_k$ and $L_k$. Instead of giving cumbersome computations, let us think geometrically referring to the figure. The signed area between $t_k$ and $t_{k+1}$ is given by 
\[
\frac{L_k(D_k+D_{k+1})-C_k(t_{k+1}-t_k-2L_k)}{2}.
\] 
As we want it to be zero we get that 
$$
C_k=\frac{L_k(D_k+D_{k+1})}{t_{k+1}-t_k-2L_k}>0.
$$   
Remembering that we need $C_k<\delta$, we can conclude choosing $L_k$ such that
$$
L_k<\frac{\delta(t_{k+1}-t_k)}{D_{k+1}+D_k+2\delta}.
$$
\end{proof}
\begin{remark}
The function $f$ is of the type $f(t) = \delta^2F(t/\delta)$ for some oscillatory function $F$. More precisely, the function $F$ satisfies Proposition \ref{prop2} with $\delta=1$. Therefore, if a smooth modification of $f$ were possible, the higher derivatives would be large.
\end{remark}

We are ready for the 
\begin{proof}[Proof of theorem \ref{main}]
Given $\delta$, consider the sequence $(t^*_k)$ coming from Lemma \ref{stand3} and the corresponding constants $\eta$ and $N$. It comes from the proof that we have $t_0=0$ and $t_N=W\in\mathbb{N}\setminus \{0\}$. 
Now consider the  corresponding sequence $(D_k)$ defined as
\begin{equation}
\begin{split}
D_k= & \frac{g}{4}(t_{k+1}^*-2t^*_k+t^*_{k-1}) \mbox{ for } 1\leq k \leq N-1  \\
D_N= & D_0=  \eta.
\end{split}
\end{equation}
From condition 2 and 3 in Lemma \ref{stand3} we have
$$
0\leq D_k\leq\delta
$$
for every $0\leq k \leq N-1$. Thus we can apply Proposition \ref{prop2} to the sequences $(t^*_k)_{0\leq k \leq N-1}$ and $(D_k)_{0\leq k \leq N}$ to get the corresponding function $\bar{f}$. Now consider the corresponding map $P_{\bar{f}}$
\begin{equation}
\begin{sistema}
t_1=t_{0}+\frac{2}{g}v_{0}-\frac{2}{g}\bar{f}[t_1,t_{0}]
\\
v_1=v_{0}+2\dot{\bar{f}}(t_1)-2\bar{f}[t_1,t_{0}].
\end{sistema}
\end{equation}
Let $(\tau_k^*,\nu^*_k)$ the orbit with initial condition
$$
(t_0,v_0)=(t_0^*,\frac{g(t^*_1-t^*_0)}{2}).
$$ 
Remembering conditions 1 and 2 of proposition \ref{prop2} we have that $(\tau_k)=(t^*_k)$  and the corresponding sequence $(t_k^*,v_k^*)$ is an orbit of $P_{\bar{f}}$ satisfying the hypothesis of Proposition \ref{stand2}. Condition 3 of Lemma \ref{stand3} concludes the proof.   
 \end{proof}

\bibliographystyle{plain}
\bibliography{biblio4}

\end{document}